\documentclass[11pt]{amsart}

\usepackage{amsmath, amssymb, xypic}

\DeclareMathOperator{\Irr}{Irr}

\newcommand{\bbN}{{\mathbb N}}
\newcommand{\bbC}{\mathbb C}

\DeclareMathOperator{\proj}{proj}

\newcommand{\cB}{\mathcal B}

\newcommand{\calD}{\mathcal D}
\newcommand{\bbP}{\mathbb P}
\newcommand{\e}{\varepsilon}
\newtheorem{thm}{Theorem}

\newtheorem{coro}[thm]{Corollary}

\newtheorem{lemma}[thm]{Lemma}

\newtheorem{prop}[thm]{Proposition}

\theoremstyle{definition}

\newtheorem{problem}[thm]{Problem}

\newcounter{my_enumerate_counter}
\newcommand{\pushcounter}{\setcounter{my_enumerate_counter}{\value{enumi}}}
\newcommand{\popcounter}{\setcounter{enumi}{\value{my_enumerate_counter}}}

\newcommand{\FileName}[1]{\thanks{Filename: {\tt #1}}}

\newcommand{\cU}{{\mathcal U}}

\DeclareMathOperator{\Ad}{Ad}

\title{A dichotomy for the Mackey Borel structure}

\author{Ilijas Farah}

\address{Department of Mathematics and Statistics\\
York University\\
4700 Keele Street\\
North York, Ontario\\ Canada, M3J 1P3}
\address{Matemati\v cki Institut\\ Kneza Mihaila 34\\ 11\,000 Beograd\\ Serbia}

\urladdr{http://www.math.yorku.ca/$\sim$ifarah}

\email{ifarah@mathstat.yorku.ca}

\thanks{The work reported in this note was done in July 2008 while I was visiting IHES and it was presented at a mini-conference in set theory at the Institut Henri Poincar\'e
in July 2008.}

\FileName{2008g25-mackey.tex}

\date{\today}


\begin{document}

\begin{abstract} 
We prove that the equivalence of pure states  of a separable C*-algebra is
either smooth or it continuously reduces $[0,1]^{\bbN}/\ell_2$ and
it therefore cannot be classified by countable structures. 
The latter  was independently proved by Kerr--Li--Pichot by using different methods. 
 We also  give some remarks on 
a 1967 problem of Dixmier. 
\end{abstract} 
\maketitle

If $E$ and $F$
are Borel equivalence relations on Polish spaces $X$ and $Y$,
respectively, then we say that $E$ is \emph{Borel reducible} to $F$
(in symbols, $E\,\leq_B\, F$) if there is a Borel-measurable map
$f\colon X\to Y$ such that for all $x$ and $y$ in $X$ we have $x E y
$ if and only if $f(x) F f(y)$. 
A Borel equivalence relation $E$ is \emph{smooth} if it is
Borel-reducible to the equality relation on some Polish space. Recall
that $E_0$ is the equivalence relation on $2^\bbN$ defined by $x E_0
y$ if and only if $x(n)=y(n)$ for all but finitely many $n$. The
\emph{Glimm--Effros} dichotomy (\cite{HaKeLo:GE}) states that a Borel
equivalence relation $E$ is either smooth or $E_0\,\leq_B\, E$.

One of the  themes of the abstract classification theory is
measuring  relative complexity of classification problems from mathematics
(see e.g.,  \cite{HjKe:Recent}). 
One can formalize the notion of `effectively classifiable by countable structures' in terms
of the relation $\leq_B$ and a natural Polish space of structures based on $\bbN$
in a natural way.  
In \cite{Hj:Book} Hjorth introduced the notion 
of turbulence for orbit equivalence relations and proved that an orbit equivalence
relation given by a turbulent action cannot be effectively classified by countable structures. 

The idea that there should be a small set $\cB$ of Borel  equivalence relations 
not classifiable by countable structures such that for every Borel equivalence 
relation $E$ not classifiable by countable structures there is $F\in \cB$ such that 
$F\,\leq_B \,E$ was put forward in \cite{HjKe:New} and, in a revised form, in~\cite{Fa:c_0}. 
In this note we prove a dichotomy for a class of Borel equivalence relations 
corresponding to the spectra of C*-algebras by showing that one of the standard 
turbulent orbit equivalence relations, $[0,1]^{\bbN}/\ell_2$, 
 is Borel-reducible to every non-smooth spectrum.

\subsection*{States} 
All undefined notions from the theory of C*-algebras
and more details can be found in \cite{Black:Operator} or in \cite{FaWo:Set}. 
Consider a separable C*-algebra $A$. 
Recall that a functional $\phi$ on $A$ is \emph{positive} if it sends 
every positive operator in $A$ to a positive real number. 
 A positive functional is a \emph{state} 
if it is of norm $\leq 1$.  The states form a compact convex set, and the 
extreme points of this set are the \emph{pure states}. 
The space of pure states on $A$,
denoted by $\bbP(A)$, equipped with the weak*-topology, is a Polish
space (\cite[4.3.2]{Pede:C*}).

A C*-algebra $A$ is \emph{unital} if it has the multiplicative identity. 
Otherwise, we define the \emph{unitization} of $A$, $\tilde A$, 
the canonical unital 
C*-algebra that has $A$ as a maximal ideal and such that 
the quotient $\tilde A/A$ is isomorphic to $\bbC$ (see \cite[Lemma~2.3]{FaWo:Set}). 
 If $u$ is a unitary in $A$ (or $\tilde A$) then 
 \[
( \Ad u) a=uau^*
\]
defines an inner automorphism of $A$. 
 
  Two pure states $\phi$ and
$\psi$ are equivalent, $\phi\sim_A\psi$,  if there exists a
 unitary $u$ in $A$ (or  $\tilde A$) such that
$\phi=\psi\circ\Ad u$.

\begin{thm} \label{T1} Assume $A$ is a separable C*-algebra. Then 
 $\sim_A$ is either smooth or there is a continuous map 
\[
\Phi\colon [0,1]^{\bbN}\to \hat A
\]
such that $\alpha-\beta\in \ell_2$ if and only if $\Phi(\alpha)\sim_A \Phi(\beta)$. 
\end{thm}

\begin{coro}  Assume $A$ is a separable C*-algebra. Then either $\sim_A$ is  smooth 
or it cannot be classified by countable structures. 
\end{coro} 

\begin{proof} 
By \cite{Hj:Book} it suffices to show that a turbulent orbit equivalence relation 
is Borel-reducible to $\sim_A$ if $\sim_A$ is not smooth. 
The equivalence relation $[0,1]^{\bbN}/\ell_2$ is well-known to 
be turbulent (e.g., \cite{HjKe:New}) 
and the conclusion follows by Theorem~\ref{T1}. 
\end{proof} 

This result was independently  
proved in \cite[Theorem~2.8]{KeLiPi:Turbulence} by directly showing the turbulence. 
As pointed out in \cite[\S3]{KeLiPi:Turbulence}, it implies an analogous result of 
Hjorth (\cite{Hj:Non-smooth}) on irreducible representations of discrete groups, 
as well as its strengthening to locally compact groups. 

\section{Proof of Theorem~\ref{T1}}

Recall that the CAR  (Canonical Anticommutation Relations) algebra (also know as 
the Fermion algebra, or $M_{2^\infty}$) is defined as the infinite tensor product
\[
M_{2^\infty}=\bigotimes_{n\in \bbN} M_2(\bbC)
\]
where $M_2(\bbC)$ is the algebra of $2\times 2$ matrices. 
Alternatively, one may think of $M_{2^\infty}$ as the direct limit of $2^n\times 2^n$ matrix 
algebras $M_{2^n}(\bbC)$ for $n\in \bbN$.

The following 
analogue of the Glimm--Effros dichotomy is an immediate consequence
of \cite{Glimm:On} (Notably, the key combinatorial device in the proof
of \cite{HaKeLo:GE} comes from Glimm). 

\begin{prop} \label{P1} If $A$ is a  separable C*-algebra then exactly one
of the following applies.
\begin{enumerate}
\item $\sim_A$ is smooth.
\item $\sim_{M_{2^\infty}}\,\leq_B\,\sim_A$. \qed
\end{enumerate}
\end{prop}

We shall prove that $\sim_{M_{2^\infty}}$ is turbulent in the sense
of Hjorth.

\begin{lemma} \label{L1} If $\xi$ and $\eta$ are unit vectors in $H$
then
\begin{equation}
\inf\{\|I-u\|: u\text{ unitary in $\cB(H)$ and }(u\xi|\eta)=1\}
=\sqrt{2(1-|(\xi|\eta)|)}.\tag{*}
\end{equation}
\end{lemma}

\begin{proof} Let $t=(\xi|\eta)$.
Let $\xi'=\frac 1{\|\proj_{\bbC\eta}\xi\|}\proj_{\bbC\eta}\xi$. Then the square of the left-hand side of  (*)
 is greater than or equal
to"
\[
\|\xi-\xi'\|^2=\|\xi\|^2+\|\xi'\|^2-\frac
2{(\xi|\eta)}(\xi|(\xi|\eta)\eta)=2-2|t|.
\]
For $\leq$ let $\zeta$ be the unit vector orthogonal to $\xi$ such
that
\[
\eta=t\xi+\sqrt{1-t^2}\zeta
\]
and let $u$ be the unitary given by $\begin{pmatrix} t &
-\sqrt{1-t^2}\\ \sqrt{1-t^2} & t\end{pmatrix}$ on the span of $\xi$
and $\zeta$ and identity on its orthogonal complement. Then
$u\xi=\eta$ and a straightforward  computation gives $\|I-u\|^2=2-2t$
as required. \end{proof}

If $\xi$ is a unit vector in a Hilbert space then by $\omega_\xi$ we
denote the vector state $a\mapsto (a\xi|\xi)$.
If $\xi_i$ is a unit vector in $H_i$ for  $1\leq i\leq m$ then $\xi=\bigotimes_{i=1}^m \xi_i$
is a unit vector in $H=\bigotimes_{i=1}^m \xi_i$ and $\omega_\xi$ is a vector state on $\cB(H)$.

\begin{lemma}\label{L2} If $H_i$  is a Hilbert space
and $\xi_i,\eta_i$ are unit vectors in $H_i$ for $1\leq i\leq m$ then
\begin{multline*}
\textstyle\inf\{\|I-u\|: u\text{ unitary and
}\omega_{\bigotimes_{i=1}^m \xi_i}=\omega_{\bigotimes_{i=1}^m
\eta_i}\circ\Ad u\}\\
= 2\sqrt{2(1-\textstyle\prod_{i=1}^m|(\xi_i|\eta_i)|)}.
\end{multline*}
\end{lemma}

\begin{proof} The case when $m=1$ follows from Lemma~\ref{L1} and the fact that $\omega_\xi=\omega_{\alpha\xi}$ when $|\alpha|=1$.
Since $(\bigotimes_{i=1}^m\xi_i|\bigotimes_{i=1}^m
\eta_i)=\prod_{i=1}^m (\xi_i|\eta_i)$, the general case is an
immediate consequence of Lemma~\ref{L1}.
\end{proof}

\begin{thm} \label{T5} There is a continuous map $\Phi\colon (-\frac \pi2,
\frac \pi2)^\bbN\to \bbP(M_{2^\infty})$ such that for all
$\vec\alpha$ and $\vec\beta$ in the domain we have
\[
\sum_n (\alpha_n-\beta_n)^2<\infty\Leftrightarrow
\Phi(\vec\alpha)\sim_{M_{2^\infty}}\Phi(\vec\beta).
\]
\end{thm}

\begin{proof} 
Consider the standard representation of $M_2(\bbC)$ on $\bbC^2$. Then 
the pure states of $M_2(\bbC)$ are of 
the form $\omega_{(\cos \alpha,\sin\alpha)}$ for $\alpha\in (-\frac \pi2,\frac\pi 2)$. 

Let $\Phi(\vec\alpha)=\bigotimes_{n=1}^\infty
\omega_{(\cos\alpha_n, \sin\alpha_n)}$. This map is continuous: If
$a\in M_{2^\infty}$ and $\e>0$, fix $m$ and $a'\in M_{2^m}$ such that
$\|a-a'\|<\e/2$. Then $\Phi(\vec\alpha)(a')$ depends only on
$\alpha_j$ for $j\leq m$, and in a continuous fashion.

Recall that for $0<t_j<1$ we have $\prod_{j=1}^\infty t_j>0$ if and only
if $\sum_{j=1}^\infty (1-t_j)<\infty$. Therefore
\[
\sum_{n=1}^\infty(\alpha_n-\beta_n)^2<\infty \Leftrightarrow
\sum_{n=1}^\infty \sin^2\left(\frac {\alpha_n-\beta_n}2\right)<\infty
\Leftrightarrow \prod_{n=1}^\infty \cos(\alpha_n-\beta_n)>1.
\]
Assume $\prod_{n=1}^\infty\cos(\alpha_n-\beta_n)>0$. In the $n$-th
copy of $M_2$ in $M_{2^\infty}=\bigotimes_{n=1}^\infty M_2$ pick a
unitary $u_n$ such that
\[
\|1-u_n\|<\sqrt{2(1-|\cos(\alpha_n-\beta_n)|)}
\]
 and
$u_n(\cos\alpha_n, \sin\alpha_n)=(\cos\beta_n,\sin\beta_n)$. Note
that
\[
((\cos\alpha_n,\sin\alpha_n)|(\cos\beta_n,\sin\beta_n))=\cos(\alpha_n-\beta_n).
\]
Let $v_n=\bigotimes_{j=1}^n u_j$. Then $v_n$ for $n\in \bbN$ form a
Cauchy sequence, because $v_m-v_{m+n}=v_m(1-\bigotimes_{j={m+1}}^n
u_j)$ and therefore
\[
\textstyle\|v_m-v_n\|<\sqrt{2(1-\prod_{j=m}^\infty \cos(\alpha_j-\beta_j))}.
\]
Let  $v\in M_{2^\infty}$ be the limit of this Cauchy sequence. Since
for each $m$ and $a\in M_{2^m}$ we have
$\Phi(\vec\alpha)(a)=\Phi(\vec\beta)(v_n a v_n^*)$ for any $n\geq m$,
we have $\Phi(\vec\alpha)=\Phi(\vec\beta)\circ \Ad v$.

Now assume $\Phi(\vec\alpha)\sim_{M_{2^\infty}} \Phi(\vec\beta)$ and, for the sake
of obtaining a contradiction, that
$\prod_{n=1}^\infty\cos(\alpha_n-\beta_n)=0$.
 There is $m$ and a unitary $u\in
M_{2^m}$ such that
\[
\|\Phi(\vec\alpha)-\Phi(\vec\beta)\circ \Ad u\|<\frac 12.
\]
(by e.g.,  \cite{Glimm:On}).  However, we can find $n>m$ large enough
so that with $\xi_n=\bigotimes_{j=m}^n (\cos\alpha_j, \sin\alpha_j)$
and $\eta_n=\bigotimes_{j=m}^n(\cos\beta_j,\sin\beta_j)$ the quantity
\[
(\xi_n|\eta_n)=\prod_{j=m}^n\cos(\alpha_j-\beta_j)
\]
is as close to zero as desired. Then
$\|\omega_{\xi_n}-\omega_{\eta_n}\|$ is as close to 2 as desired,
since $a_n=\proj_{\bbC \xi_n}-\proj_{\bbC \eta_n}$ has norm close to
1 and $\omega_{\xi_n}(a_n)$ is close to 1 while
$\omega_{\eta_n}(a_n)$ is close to $-1$.
\end{proof}

\begin{proof}[Proof of Theorem~\ref{T1}]
 Assume $\sim_A$ is not smooth. 
The conclusion follows by  Glimm's Proposition~\ref{P1} 
and  Theorem~\ref{T5}.   \end{proof} 

\section{Concluding remarks}

We note that  the class of equivalence relations 
corresponding to spectra of C*-algebras is restrictive in another sense. 
 The following proposition was probably well-known (cf. \cite[Corollary~1.3]{Hj:Non-smooth}). 
 
\begin{prop} If $A$ is a separable C*-algebra then the
relation $\phi\sim_A \psi$ on $\bbP(A)$ is $F_\sigma$.
\end{prop}

\begin{proof} By replacing $A$ with its unitization if necessary we may assume $A$ is unital.
Fix a countable dense set $\cU$ in the unitary group of $A$ and a
countable dense set $\calD$ in $A_{\leq 1}$. We claim that
\[
\phi\sim_A \psi\Leftrightarrow (\exists u\in \cU)(\forall a\in
\calD)|\phi(a)-\psi(uau^*)|<1.
\]
Assume $\phi\sim_A\psi$ and fix $v$ such that $\phi=\psi\circ \Ad v$.
If $u\in \cU$ is such that $\|v-u\|<1/2$ then
\[
|\psi(uau^*-vav^*)|=|\psi((u-v)au^*-va(u^*-v^*))|<1
\]
for all $a\in A_{\leq 1}$.

Now assume $u\in \cU$ is such that $|\phi(a)-\psi(uau^*)|<1$ for all
$a\in \calD$. Then $\|\phi-\psi\circ \Ad u\|<2$ and by \cite{GliKa}
we have $\phi\sim_A\psi$.
\end{proof}

For a Hilbert space $H$ by $\cB(H)$ we denote the algebra of its bounded linear operators. 
Let  $\pi_1\colon A\to \cB(H_1)$ and $\pi_2\colon A\to \cB(H_2)$ be representations of~$A$. 
 We say $\pi_1$ and $\pi_2$ are
\emph{(unitarily) equivalent} and write  $\pi_1\sim \pi_2$ if there
is a Hilbert space isomorphism $u\colon H_1\to H_2$ such
that the diagram
\[
\diagram & \cB(H_1)\ddto^{\Ad u}\\
A\urto^{\pi_1} \drto_{\pi_2} & &  \Ad u(a)=uau^*\\
& \cB(H_2)
\enddiagram
\]
commutes. 

A representation of $A$ on some Hilbert space $H$ is \emph{irreducible} if there are no  nontrivial 
 closed  subspaces of $H$ invariant under the image of $A$.   
The spectrum of $A$, denoted by $\hat
A$, is the space of all equivalence classes of irreducible
representations of $A$. The GNS construction associates a
representation $\pi_\phi$ of $A$ to each  state
$\phi$ of $A$ (see e.g., \cite[Theorem~3.9]{FaWo:Set}). 
Moreover, $\phi$ is pure if and only if $\pi_\phi$ is irreducible (\cite[Theorem~3.12]{FaWo:Set})
and for pure states $\phi_1$ and $\phi_2$ we have that $\phi_1$ and $\phi_2$ are equivalent
if and only if $\pi_{\phi_1}$ and $\pi_{\phi_2}$ are equivalent (\cite[Proposition~3.20]{FaWo:Set}).

Fix  a separable C*-algebra $A$. 
Let $\Irr(A,H_n)$ denote  the space of irreducible representations
of $A$ on a Hilbert space $H_n$ of dimension $n$ for $n\in \bbN\cup
\{\aleph_0\}$. 
 Each $\Irr(A,H_n)$ is a Polish space with respect to
the weakest topology making all functions $\Irr(A,H_n)\ni \pi\mapsto
(\pi(a)\xi|\eta)\in \bbC$, for $a\in A$ and $\xi,\eta\in H_n$,
continuous. In other words, a net $\pi_\lambda$ converges to $\pi$ if
and only if $\pi_\lambda(a)$ converges to $\pi(a)$ for all $a\in A$.
 Since  $A$ is separable,  each irreducible
representation of $A$ has range in a separable Hilbert space, and
therefore $\hat A$ can be considered as a quotient space of the
direct sum of $\Irr(A,H_n)$ for $n\in \bbN\cup \{\aleph_0\}$. 
Therefore $\hat A$ carries a Borel structure (known as the
\emph{Mackey Borel structure}) inherited from a Polish space.
  For \emph{type I} C*-algebras (also
called \emph{GCR} or \emph{postliminal}) this space is a standard Borel space.
(All of these notions are explained in \cite[\S4]{Arv:Invitation}.)

Since pure states correspond to irreducible representations, we can identify 
 the Mackey Borel structure of $A$ with  a
 $\sigma$-algebra of sets in $\hat A$. It is easy to check that this $\sigma$-algebra 
 consists exactly of those sets  whose preimages in $\bbP(A)$
 are Borel subsets in $\bbP(A)$.

Glimm proved (\cite{Glimm:On}, \cite[\S 6.8]{Pede:C*})  that the Mackey Borel structure of
a C*-algebra $A$ is \emph{smooth} (i.e., isomorphic to a standard
Borel space) if and only if $A$ is a type I C*-algebra.
Proposition~\ref{P1} is a consequence of this result.

\begin{problem}[Dixmier, 1967]\label{P.Mackey-Borel}
Is the Mackey Borel
structure on the spectrum of a simple separable C*-algebra always the same when
it is not standard?
\end{problem}

 G. Elliott  generalized  Glimm's 
 result and proved 
 that the Mackey Borel structures  of simple AF algebras are isomorphic
 (\cite{Ell:Mackey}).  (A C*-algebra is an AF (approximately finite) algebra if it is a direct limit
 of finite-dimensional algebras.)
One  reformulation of Elliott's result is that for any two
simple separable AF algebras $A$ and $B$ there is a Borel isomorphism
$F\colon \bbP(A)\to \bbP(B)$ such that $\phi\sim_A \psi$ if and only
if $F(\phi)\sim_B F(\psi)$ (see \cite[\S6]{Ell:Mackey}).
Also,   \cite[Theorem 2]{Ell:Mackey}
implies that if $A$ is a simple separable AF algebra and $B$ is a
non-Type I simple separable algebra we have $\sim_A\,\leq_B\,
\sim_B$.

With this definition the quotient structure
Borel$(\bbP(A))/\sim_A$ is isomorphic to the Mackey Borel structure
of $A$. 
Note that  $\sim_A$ is smooth exactly when the Mackey Borel structure of
$A$ is smooth.

 Note that Mackey Borel structures
of $A$ and $B$ of separable C*-algebras are isomorphic if and only if there is a Borel
isomorphism $f\colon \hat X\to \hat X$ such that $\pi_1\sim_A \pi_2$ if and
only if $f(\pi_1)\sim_Bf(\pi_2)$. Hence Problem~\ref{P.Mackey-Borel} is
rather close in spirit to the theory of Borel equivalence relations.

N. Christopher Phillips suggested more general problems about the Mackey
Borel structure of simple separable C*-algebras, motivated by his discussions
with Masamichi Takesaki. There are two (related) kinds of questions: Can one do
anything sensible, and, from the point of view of logic, how bad is the
problem?

\begin{problem}\label{P.Borel.2}
Does the complexity of the Mackey Borel structure of a simple separable
C*-algebra increase as one goes from nuclear C*-algebras to exact ones to ones
that are not even exact?
\end{problem}

For definitions of nuclear and exact C*-algebras see e.g., \cite{Black:Operator}. 

\begin{problem} Assume $A$ and $B$ are C*-algebras and $\sim_A$ is
 Borel-reducible to $\sim_B$. What does this fact imply about the
relation between $A$ and $B$? \end{problem}

\providecommand{\bysame}{\leavevmode\hbox to3em{\hrulefill}\thinspace}
\providecommand{\MR}{\relax\ifhmode\unskip\space\fi MR }
\providecommand{\MRhref}[2]{%
  \href{http://www.ams.org/mathscinet-getitem?mr=#1}{#2}
}
\providecommand{\href}[2]{#2}

\end{document}